\documentclass[12pt, titlepage]{article}

\usepackage{eurosym}
\usepackage{amsmath,amssymb,amsthm,amsfonts}
\usepackage[mathscr]{eucal}
\usepackage{graphicx}
\usepackage{epstopdf}
\usepackage{dsfont}

\usepackage{subcaption}
\usepackage{amsmath}
\usepackage{amsfonts}
\usepackage{amssymb}
\usepackage{epsfig}
\usepackage{epstopdf}%

\usepackage{graphicx}

\numberwithin{equation}{section} 

\newtheorem{theorem}{Theorem}[section]

\newtheorem{proposition}{Proposition}[section]

\newtheorem{remark}{Remark}

\def\XX{X}
\def\xx{\boldsymbol{x}}

\def\FF{\boldsymbol{F}}

\def\RRR{\boldsymbol{R}}

\def\FFF{\mathcal{F}}

\def\QQQ{\mathbb{Q}}
\def\RR{\mathbb{R}}

\def\ff{\boldsymbol{f}}

\def\design{\mathcal{S}_m}
\def\expval{\mathbf{E}}
\def\QQQ{\mathcal{Q}_m}

\DeclareMathOperator{\diag}{diag}
\DeclareMathOperator\Binomial{Binomial}

\begin{document}

\author{Roberto Fontana \thanks{
Department of Mathematical Sciences, Politecnico di Torino, roberto.fontana@polito.it} and Patrizia
Semeraro\thanks{
Department of Mathematical Sciences, Politecnico di Torino, patrizia.semeraro@polito.it} }

\date{}
\title{Characterization  of multivariate Bernoulli  distributions with given margins} \maketitle

\begin{abstract}
We express each Fr\'echet class of multivariate Bernoulli distributions with given margins as the
convex hull of a set of densities, which  belong to the same Fr\'echet class. This characterisation allows us to establish whether a given correlation matrix is compatible with the assigned margins and, if it is, to easily construct one of the corresponding joint densities.
We  reduce the problem of finding a density belonging to a Fr\'echet class and with given correlation matrix to the
solution of a linear system of equations.  Our methodology  also provides
the bounds that each correlation must satisfy to be compatible with the assigned margins. An algorithm and its use in some examples is shown.


\end{abstract}

\section{Introduction}

Dependent binary variables play a key role in many important scientific fields such as clinical trials and health studies.
The problem of the simulation of correlated binary data is extensively addressed in the statistical literature, e.g. \cite{chaganty2006range}, \cite{haynes2016simulating}, \cite{shults2016simulating} and \cite{kang2001generating}. Simulation studies are a useful tool for analysing extensions or alternatives to current estimating methodologies, such as generalised linear mixed models, or for the evaluation of statistical procedures for marginal regression models (\cite{qaqish2003family}). The simulation problem consists of constructing multivariate distributions for given Bernoulli marginal distributions and a given correlation matrix ${\rho}$.
Frequently,  assumptions are made about the correlation structure. Probably the most common is equicorrelation, e.g. \cite{chaganty2006range}. A popular approach also uses working correlation matrices (\cite{liang1986longitudinal} and  \cite{zeger1986longitudinal}),   such as  first order moving average correlations or first order autoregressive correlations (\cite{oman2009easily} and references therein). An important issue for these simulation procedures is the compatibility of marginal binary variables and their correlations, since problems may arise when the margins and the correlation matrix are not compatible (\cite{crowder1995use}, \cite{rao2004efficiency} and \cite{chaganty2006range}). The range of admissible correlation matrices for binary variables is well known in the bivariate case.  This problem has been widely identified in the literature, but, to the best of our knowledge no effective solution exists for multivariate binary distributions with more than three variables (\cite{chaganty2006range}).

We propose a new but simple methodology to characterise Bernoulli variables belonging to a given Fr\'echet class, i.e. with given marginal distributions. This characterisation allows us to establish  whether a given correlation matrix is compatible with the assigned margins and, if  it is, to easily construct one of the corresponding joint densities. It also provides  the bounds that each correlation must satisfy to be compatible with the assigned margins. Furthermore, if the correlation structure and the margins are not compatible, we can find a new correlation matrix which is close to the  desired one but compatible with the given margins. It is worth noting that this methodology  puts no restriction either on the number of variables or on the correlation structure. It also provides a new  computational procedure to simulate multivariate distributions of binary variables with assigned margins and given moments.

The proposed methodology is based on a polynomial representation of all the multivariate Bernoulli distributions of a given Fr\'echet class, i.e. of all the distributions with fixed Bernoulli margins. This representation is linked to the  Farlie-Gumbel-Morgesten copula (\cite{nelsen2006introduction}). It allows us to write each Fr\'echet class  as the convex hull of  the \emph{ray densities}, which are densities that belong to the Fr\'echet class under consideration. By so doing, the problem of finding one distribution with given moments in a Fr\'echet class is reduced to the solution of a linear system of equations.


\section{Preliminaries}\label{Preliminary}

Let $\mathbb{F}_m$ be the set of $m$-dimensional distributions which have Bernoulli univariate marginal distributions. Let us consider the Fr\'echet class  $\mathcal{F}(p_1, \dots, p_m)\subseteq\mathbb{F}_m$ of distribution functions in $\mathbb{F}_m$ which have the same Bernoulli marginal distributions $B(p_i), 0<p_i<1, i=1,\ldots,m$.
If  $\XX=(X_1, \dots, X_m)$ is  a random vector with joint distribution in $\mathcal{F}(p_1, \dots, p_m)$, we denote
\begin{itemize}
	\item its cumulative distribution function by $F_p$ and its density function by $f_p$ where $p=(p_1,\ldots,p_m)$;
	\item the column vector which contains the values of $F_p$ and $f_p$ over $\design:=\{0, 1\}^m$, with a small abuse of notation, still by $\FF_p = (F_p(x):x\in\design)$ and $\ff_p = (f_p(x):x\in\design)$ respectively; we make the non-restrictive hypothesis that $\design$ is ordered according to the reverse-lexicographical criterion;
	\item the marginal cumulative distribution function and the marginal density function of $X_i$ by $F_{p,i}$ and $f_{p,i}$  respectively,  $i=1,\ldots,m$;
	\item the values $f_{p,i}(0) \equiv F_{p,i}(0)$ and $f_{p,i}(1)$ by $q_i$ and $p_i$ respectively,  $i=1,\ldots,m$.
\end{itemize}

We observe that $q_i=1-p_i$ and that the expected value of $X_i$ is $p_i$, $\expval[X_i]=p_i$, $i=1,\dots,m$.

Given two matrices  $A\in \mathcal{M}(n\times m)$ and $ B\in \mathcal{M}(d\times l)$ the matrix $A\otimes B\in \mathcal{M}(nd\times ml)$ indicates their Kronecker product and $A^{\otimes n}$ is $\underbrace{A\otimes \ldots \otimes A}_{n \text{ times}}$.

If we consider a Bernoulli variable $B(\tau), 0<\tau<1$, with $F_\tau$ and $f_\tau$ as cumulative and density function respectively, the following holds
\begin{equation*}
\left(
\begin{array}{c}
f_\tau(0) \\
f_\tau(1)
\end{array}
\right)=D \cdot
\left(
\begin{array}{c}
F_\tau(0) \\
F_\tau(1)
\end{array}
\right)
\label{eq:F2f}
\end{equation*}
where $D=\left(
\begin{array}{rr}
1 & 0 \\
-1 & 1
\end{array}
\right)
$ is the \emph{difference matrix}.

It follows that given $F_p$ and $f_p$ in $\mathcal{F}(p_1, \dots, p_m)$ we have
\begin{equation}
\label{eq:F2flarge}
\ff_p=D^{\otimes m} \FF_p.
\end{equation}

Finally we can write $\ff_p \in \FFF(p_1, \dots, p_m)$, $\FF_p \in \FFF(p_1, \dots, p_m)$ and $\XX\in \FFF(p_1, \dots, p_m)$.

\section{Construction of multivariate Bernoulli distributions with given margins}\label{FGM}
%
We give a polynomial and matrix representation of \emph{all} the $F_p \in \FFF(p_1, \dots, p_m)$. We make the non-restrictive hypothesis that  $\{q_1,1\} \times \ldots \times \{q_m,1\}$ is ordered according to the reverse-lexicographical criterion. We denote $\{q_1,1\} \times \ldots \times \{q_m,1\}$ by $\QQQ$.
%

\begin{theorem}\label{cumrep}
Any distribution $F_p\in \mathcal{F}(p_1, \dots, p_m)$ admits the following representation over $\QQQ$
\begin{equation*}\label{cum}
\FF_p = \Lambda_p U_p \theta
\end{equation*}
where
 $\Lambda_p= \diag ( q_1^{(1-\alpha_1)}\cdot \ldots \cdot q_m^{(1-\alpha_m)}  , (\alpha_1,\ldots,\alpha_m) \in \design)$,
$U_p = U_{p_1} \otimes \ldots \otimes U_{p_m}$,
$U_{p_i} =\left(
\begin{array}{rr}
1& 1-q_i \\
1 & 0\\
\end{array}%
\right), i=1,\ldots,m$ and $\theta = (\theta_0, \theta_m, \theta_{m-1}, \theta_{m,m-1}, \ldots, \theta_{12\ldots m})$.

 Necessary conditions for $F_p$ being a distribution are $\theta_0=1$ and $\theta_i=0, i=1,\ldots,m$.
\end{theorem}

\begin{proof}
Given $u=(u_1,\ldots,u_m) \in \QQQ$ let us define $$g(u)=\left( \prod_{i=1}^m u_i\right)\big(\theta_0+\sum_{j=1}^m\theta_{j}(1-u_j)+\sum_{1\leq j<k\leq m}\theta_{jk}(1-u_j)(1-u_k)+\dots+\theta_{12\dots m}\prod_{i=1}^m(1-u_i)\big)$$ and the row vectors $a_i=(1, \; 1-u_i),\; i=1,\ldots,m$. We can write $g(u)\in\RR$ as
\[
g(u)=\left(\prod_{i=1}^nu_i\right) \left( a_1 \otimes \ldots \otimes a_m \right) \left(
\begin{array}{c}
\theta_0 \\
\theta_m \\
\theta_{m-1} \\
\ldots \\
\theta_{12\ldots m} \\
\end{array}
\right).
\]
Considering all the $u \in \QQQ$ we get the $2^m$-vector $\left(g(u), u\in\QQQ \right) = \Lambda_p U_p \theta$.

We observe that the determinant of $U_{p_i} =\left(
\begin{array}{rr}
1& 1-q_i \\
1 & 0\\
\end{array}%
\right)$ is $\det(U_{p_i})=-p_i \neq 0$. It follows that the determinant of $U_p$, which is $(p_1\cdot \ldots \cdot p_m)^2$,  is also different from zero. Being the determinant of $\Lambda_p\neq0$ we get that the determinant of $\Lambda_p U_p$ is different from zero. It follows that the rank of $\Lambda_p U_p$ is $2^m$ and then any vector $y \in \RR^{2^m}$ and in particular any distribution $F_p$ can be written as $\FF_p = \Lambda_p U_p \theta$.

If $F_p$ is a distribution in $\mathcal{F}(p_1, \dots, p_n)$, the vector parameter $\theta$  must satisfy the following necessary conditions:
\begin{enumerate}
	\item $\theta_0=1$. The condition $F_p(1,\ldots,1)=1$ implies $\theta_0=1$, since  $F_p(1,\ldots,1)=\theta_0$;
	\item $\theta_i=0, i=1,\ldots,m$. The condition $F_p(1,\ldots1,0,1,\ldots,1)=q_i$ implies $\theta_i=0, i=1,\ldots,m$,  since $F_p(1,\ldots1,0,1,\ldots,1)=q_i(1+\theta_i(1-q_i))$.
\end{enumerate}
\end{proof}

\begin{remark}
Under the necessary assumptions $\theta_0=1$ and $\theta_i=0,\, i=1,\ldots,m$, the polynomial function $g(u)$ in Theorem \ref{cumrep} is the restriction of the well-known Farlie-Gumbel-Morgesten copula $C(u)$ to $\QQQ$:
$$C(u):=\left(\prod_{i=1}^m u_i\right)\big(1+\sum_{1\leq j<k\leq n}\theta_{jk}(1-u_j)(1-u_k)+\dots+\theta_{12\dots m}\prod_{i=1}^m(1-u_i)\big), \,\,\, u\in [0,1]^m.$$
Notice that the condition $\theta_0=1$ derives from  $C(1,\ldots,1)=1$ and  the condition $\theta_i=0$ is necessary since a requirement to be a copula is that $C(1,\ldots1,q_i,1,\ldots,1)=q_i$, $i=1,\ldots,m$. Our representation shows that the restriction to $\QQQ$ of the Farlie-Gumbel-Morgesten copula allows us to represent all the binary distributions with given margins, and therefore to model all the possible dependence structures of multivariate Bernoulli distributions.
\end{remark}

As a consequence of Theorem \ref{cumrep} and Equation \ref{eq:F2flarge} any density $f_p\in \mathcal{F}(p_1, \dots, p_m)$  admits the following representation over $\design$
\begin{equation}\label{dens}
\ff_p = D^{\otimes m}  \Lambda_p U_p \theta
\end{equation}
We observe that given $\ff_p \in \mathcal{F}(p_1, \dots, p_m)$ we can write it as in Eq.\eqref{dens}. Vice versa Theorem \ref{cumrep} does not provide any condition on $\theta_{i_1, \ldots, i_k}$ for $k\geq 2$ such that $D^{\otimes m}  \Lambda_p U_p \theta$ represents a density function $\ff_p$ over $\design$.

In the remaining part of this section we will provide a representation of \emph{all} the densities $f_p \in \mathcal{F}(p_1, \dots, p_m)$.

\begin{theorem}\label{radii}
Let $\ff_p \in \FFF(p_1,\dots,p_m)$. It holds that
\begin{equation}\label{prays}
\ff_p=\sum_{i=1}^{n_{\FFF}}\lambda_i\RRR_p^{(i)},
\end{equation}
where $\RRR_p^{(i)}=(R_p^{(i)}(x), x\in\design)\in \FFF(p_1,\ldots,p_m)$, $\lambda_i \geq 0$, $i=1,\dots, n_{\FFF}$ and $\sum_{i=1}^{n_{\FFF}}\lambda_i=1$.
\end{theorem}

\begin{proof}

Let us define $Y_p = D^{\otimes m}  \Lambda_p U_p$. From Eq.\eqref{dens} it holds that
\begin{equation*}
\ff_p=Y_p\Theta,
\end{equation*}
with the conditions $\theta_0=1$ and $\theta_i=0, \,\,\, i=1,\dots,m$. We can write
\begin{equation*}
\Theta=Y_p^{-1}\ff_p.
\end{equation*}
The conditions $\theta_i=0, \,\,\, i=1,\dots,m$  can be written as
\begin{equation}\label{H}
H\ff_p=0,
\end{equation}
where $H$ is the $m \times 2^m$ sub-matrix of $Y_p^{-1}$ obtained by selecting the rows corresponding to $\theta_i$, $i=1\ldots,m$.

The condition $\theta_0=1 \Leftrightarrow F_p(1,\ldots,1)=1$ is ensured by requiring that $f_p$ is a density, i.e.
\begin{enumerate}
	\item $f_p(x)\geq0$;
	\item $\sum_{\xx} f_p(x)=1$
\end{enumerate}
where $x \in \design$.


All the positive solutions $\ff_p$  of \eqref{H} have the following form:
\begin{equation*}\label{dolR}
\ff_p=\sum_{i=1}^{n_{\FFF}}\tilde{\lambda}_i\tilde{R}^{(i)}_p,  \,\,\, \tilde{\lambda}_i \geq 0,
\end{equation*}
where  $\tilde{R}^{(i)}_p=(\tilde{R}^{(i)}_{p,j}, j=1,\ldots, 2^m)\in \RR^{2^m}$, $\tilde{R}^{(i)}_{p,j} \geq 0$ and $H \tilde{R}^{(i)}_p=0$, $i=1,\ldots,n_{\FFF}$ are the \emph{extremal rays} of the cone defined by $H f_p=0$ (\cite{4ti2} and \cite{hemmecke2002computation}).

By dividing $\tilde{R}^{(i)}_p$ by the sum of its elements $\tilde{R}_{p,+}^{(i)}=\sum_{j=1}^{2^m}\tilde{R}^{(i)}_{p,j}$ we can write
\begin{equation*}
\ff_p=\sum_{i=1}^{n_{\FFF}}\lambda_i\RRR_p^{(i)},
\end{equation*}
where $\lambda_i={\tilde{\lambda}_i}{\tilde{R}^{(i)}_{p,+}}$  and $\RRR_p^{(i)}=\frac{\tilde{R}^{(i)}_p}{\tilde{R}_{p,+}^{(i)}}, \,\,\, i=1,\dots, n_{\FFF}$. It follows that  $\sum_{j=1}^{2^m}  R_{p,j}^{(i)}=1$ and that the \emph{ray density} defined as $R^{(i)}_p(x):=R_{p,j}^{(i)}$ being $x$ the $j$-th element of $\design$ belongs to $\FFF(p_1, \ldots, p_m)$, $i=1,\ldots,n_{\FFF}$.

Finally the condition   $\sum_{\xx} f_p(\xx)=1$ implies  $\sum_{i=1}^{n_{\FFF}}\lambda_i(\sum_{j=1}^{2^m}  R_{p,j}^{(i)})=\sum_{i=1}^{n_{\FFF}}\lambda_i=1$.
Then we have $\lambda_i \geq 0, i=1,\ldots,m$ and $\sum_{i=1}^{n_{\FFF}}\lambda_i=1$ and the assert is proved.
\end{proof}

Notice that Theorem \ref{radii} makes extremely easy to generate any density $\ff_p$ of the Fr\'echet class $\FFF(p_1,\dots,p_m)$. It is enough to take a positive vector $\lambda=(\lambda_1,\ldots, \lambda_{n_{\FFF}})$, such that $\sum_{i=1}^{n_{\FFF}}\lambda_i=1$,  and build $f_p=\sum_{i=1}^{n_{\FFF}}\lambda_i\RRR_p^{(i)}$.

The constraints $\expval[X_i]=p_i, \; i=1,\ldots,m$ allow us to obtain an interesting intepretation of the matrix $H$ of \eqref{H}.  We have $\expval[X_i] = \sum_{(x_1,\ldots,x_m) \in \design} x_i f_p(x_1,\ldots,x_m)$. It follows that
\begin{eqnarray*}
x_i^T f_p=p_i \\
(1-x_i)^T f_p =q_i
\end{eqnarray*}
where $x_i$ is the vector which contains the $i$-th element of $x \in \design$, $i=1,\ldots,m$. If we consider the odds of the event $X_i=1$, $\gamma_i=p_i/q_i$ we have $\gamma_i q_i - p_i=0$. We can write
\[
(\gamma_i (1-x_i)^T - x_i^T) f_p = 0.
\]
Then $H$ is simply the $m \times 2^m$ matrix whose rows, up to a non-influential multiplicative constant, are $(\gamma_i (1-x_i)^T - x_i^T)$, $i=1,\ldots,m$.

Using Theorem \ref{radii} we represent each Fr\'echet class $\FFF(p_1,\ldots,p_m)$ as the convex hull of the \emph{ray densities}. We observe that the \emph{ray densities} depend only on the marginal distributions $F_1, \dots, F_m$.

Building the \emph{ray} matrix $R_p$
\[
R_p=\left(
\begin{array}{ccc}
R_{p,1}^{(1)} & \ldots & R_{p,1}^{(n_{\FFF})} \\
& \ldots & \\
R_{p,2^m}^{(1)} & \ldots & R_{p,2^m}^{(n_{\FFF})}
\end{array}
\right)
\]
whose columns are the ray densities $\RRR_p^{(i)}, i=1,\ldots, n_{\FFF}$ we write Eq.\eqref{prays} simply as
\[
\ff_p=R_p \lambda
\]
with $\lambda=(\lambda_1,\ldots,\lambda_{n_{\FFF}}), \lambda_i \geq 0$ and $\sum_{i=1}^{n_{\FFF}}\lambda_i=1$.

In practical applications the rays $\tilde{R}^{(i)}_p$ and therefore the \emph{ray densities} $R^{(i)}_p$  can be found using the software 4ti2, \cite{4ti2}. In Section \ref{sec:examples} we will use SAS and 4ti2 to show some numerical examples.

In the next sections we will see  that the representation of $\ff_p$ as in Theorem \ref{radii} plays a key role in determining the densities with given moments.

\subsection{Moments of multivariate Bernoulli variables}
We observe that, given the Bernoulli variable $X \sim B(\tau), 0<\tau<1$ with density function $f_\tau$  we can compute the moments $\expval [X^\alpha], \alpha \in \{0,1\}$ as
\[
\expval [X^\alpha]=
\left(
\begin{array}{c}
\expval[1] \\
\expval[X] 	
\end{array}
\right )
= M
\left(
\begin{array}{c}
f_\tau(0) \\
f_\tau(1) 	
\end{array}
\right)
\]
where $M=\left(
\begin{array}{rr}
1 & 1 \\
0 & 1
\end{array}
\right )$.

It follows that given $\XX=(X_1, \dots, X_m) \in \FFF(p_1, \dots, p_m)$ with multivariate joint density $f_p$, we can compute the vector of its moments $\expval [\XX^\alpha] \equiv \expval [X_1^{\alpha_1} \cdot \ldots \cdot X_m^{\alpha_m}], \alpha=(\alpha_1,\ldots,\alpha_m) \in \design$ as
\[
\expval[\XX^\alpha]= M^{\otimes m} \ff_p.
\]
We also observe that the correlation $\rho_{ij}$ between two Bernoulli variables $X_i \sim B(p_i)$ and $X_j \sim B(p_j)$ is related to the second-order moment $\expval[X_i X_j]$ as follows
\begin{equation}\label{eq:rho_e12}
\expval[X_iX_j]=\rho_{ij} \sqrt{p_iq_ip_jq_j}+p_ip_j.
\end{equation}

\subsection{Second-order moments of multivariate Bernoulli variables with given margins} \label{sec:m2gm1}
From Theorem \ref{radii} we get
\[
\expval [\XX^\alpha] =  M^{\otimes m} \ff_p = M^{\otimes m} R_p \lambda.
\]
In particular for the second-order moments $\mu_2= \expval[\XX^\alpha: \|\alpha\|_0=2]$, where $\|\alpha\|_0=\sum_{i=1}^m \alpha_i$ we get the following result, which is crucial for the solution of the problem of simulating multivariate binary distributions with a given correlation matrix.
\begin{proposition} \label{pr:mu2}
It holds that
\begin{equation}\label{eq:rho}
\mu_2=A_{2p} \lambda
\end{equation}
where  $A_{2p}=\left(M^{\otimes m}\right)_2 R_p$ and $\left(M^{\otimes m}\right)_2$ is the sub-matrix of $M^{\otimes m}$ obtained by selecting the rows corresponding to the second-order moments, $R_p$ is the \emph{ray matrix} and $\lambda=(\lambda_1,\ldots,\lambda_{n_{\FFF}})$, $\lambda_i \geq 0, i=1,\ldots m$ and $\sum_{i=1}^{n_{\FFF}}\lambda_i=1$.
\end{proposition}

It follows that the target second-order moments are compatible with the means if they belong to the convex hull
generated by the points which are the columns of the $A_{2p}=\left(M^{\otimes m}\right)_2 R_p$ matrix.
As a direct consequence of Proposition \ref{pr:mu2} we also get the univariate bounds for the second-order moments and the correlations.

\begin{proposition} \label{pr:bounds}
For each $\alpha$, $\|\alpha\|_0=2$, the second-order moment $\mu_2^{(\alpha)}$ must satisfy the following bounds
\begin{equation}
\min A_{2p}^{(\alpha)} \leq \mu_2^{(\alpha)} \leq \max A_{2p}^{(\alpha)}
\label{eq:pbounds}
\end{equation}
and the correlations $\rho_{S(\alpha)}$ must satisfy the following bounds
\begin{equation} \label{eq:bounds4cor}
\frac{\min A_{2p}^{(\alpha)}-p_ip_j}{\sqrt{p_iq_ip_jq_j}} \leq \rho_{ij} \leq \frac{\max A_{2p}^{(\alpha)}-p_ip_j}{\sqrt{p_iq_ip_jq_j}}
\end{equation}
where $A_{2p}^{(\alpha)}$ is the row of the matrix $A_{2p}$ such that $\mu_2^{(\alpha)}=A_{2p}^{(\alpha)} \lambda$ and
$\{i,j\}=\{k: \alpha_k=1\}$.
\end{proposition}
\begin{proof}
From Proposition \ref{pr:mu2} using the the proper row of $A_{2p}$ we get
\[
\mu_2^{(\alpha)}= A_{2p}^{(\alpha)} \lambda.
\]
To prove \eqref{eq:pbounds} it is enough to observe that
\begin{enumerate}
	\item being $\lambda_i \geq0$ and $\sum_{i=1}^{n_{\FFF}}\lambda_i=1$ it follows that the minimum (maximum) value of $\mu_2^{(\alpha)}$ will be obtained choosing $\lambda$ equal to one of the $e_i$'s, where $e_i \in \{0,1\}^{n_{\FFF}}$ is the binary vector with all the elements equal to zero apart from the $i$-th which is equal to one, $i=1,\ldots,n_{\FFF}$;
	\item the product $A_{2p}^{(\alpha)} e_i$ gives the $i$-th element of $A_{2p}^{(\alpha)}$.
	\end{enumerate}
To prove \eqref{eq:bounds4cor} we simply observe that using equation \eqref{eq:rho_e12} the bounds in \eqref{eq:pbounds} can be transformed to those suitable for correlations.
\end{proof}


Now we solve the problem of constructing a multivariate Bernoulli density $f_p \in \mathcal{F}(p_1, \dots, p_m)$ with given correlation matrix ${\rho}=(\rho_{ij})_{ i,j=1,\dots,m}$. Using Equation \eqref{eq:rho_e12} we transform the desired correlations $\rho_{ij}$ into the corresponding desired second-order moments $\expval[X_iX_j], i,j=1,\ldots,m, i<j$. In this way the density $\ff_p$ with means $p_1,\ldots,p_m$ and correlation matrix $\rho$ can be built as $R_p \lambda$, where $\lambda=(\lambda_1,\ldots,\lambda_{n_\mathcal{F}}), \lambda_i \geq 0, \sum_{i=1}^{n_\mathcal{F}} \lambda_i=1$ is a solution, if it exists, of the system of equations \eqref{eq:rho}.

The space of solutions $\lambda$ of the system \eqref{eq:rho} defines the set of distributions in the Fr\'echet class with correlation matrix $\rho$. The choice of a particular solution does not modify the distributions of  the sample means and of the sample second-order moments, which depend only on $p_1,\ldots,p_m$ and $\rho$ respectively.
To explain this point let us consider a random sample $\{(X_{k1},\ldots,X_{km}), \, k=1,\ldots,N\}$ extracted from a randomly selected $m$-dimensional Bernoulli variable belonging to the  Fr\'echet class $\FFF(p_1, \ldots, p_n)$ and with given second-order moments $\mu_{ij}:=E[X_iX_j], \, i,j=1,\ldots,n$.
The sample means $\overline{X}_i, \, i=1,\ldots, m$ are $\frac{1}{N} \Binomial(N,p_i)$ and the sample second-order moments  $\overline{X_iX_j}:=\sum_{k=1}^N\frac{X_{ki}X_{kj}}{N}, \, i,j=1,\ldots,n, i<j$ are $\frac{1}{N} \Binomial(N,\mu_{ij})$.

In general different distributions which belong to the same Fr\'echet class and which have the same correlation matrix $\rho$ (or equivalently the same vector of second-order moments $\mu_2$), will have different $k$-order moments, with $k\geq 3$. This methodology offers the opportunity to choose the \emph{best} distribution according to a certain criterion. For example, as the moments of multivariate Bernoulli are always positive, it could be of interest to find one of the distributions with the smallest sum of all the moments with order greater than $2$. This problem can be efficiently solved using linear programming techniques (\cite{berkelaar2004lpsolve}). It can be simply stated as
\[
\min_{f \in \mathbb{F}_m} (1^{T} (M^{\otimes m})_{3\ldots m} f)
\]
subject to
\[
\begin{cases}
H f=0 \\
(M^{\otimes m})_{2} f= \mu_2
\end{cases}
\]
where $1$ is the vector with all the elements equal to $1$ and $(M^{\otimes m})_{3\ldots m}$ is the sub-matrix of $M^{\otimes m}$ obtained by selecting the rows corresponding to the $k$-moments, with $k \geq 3$.

As we already mentioned, from a geometrical point of view a solution of the system of equations \eqref{eq:rho} exists if and only if a point whose coordinates are the desired second-order moments belongs to the convex hull generated by the points which are the columns of the $A_{2p}=\left(M^{\otimes m}\right)_2 R_p$ matrix. If the margins and the correlation matrix are not compatible, the system \eqref{eq:rho} does not have any solution. In this case it is possible to search for a feasible $\rho^*$ which is the correlation matrix closest to the desired $\rho$, according to a chosen distance.

Finally it is worth noting that the method can be applied to the moments of order greater than $2$ or to any selection of moments by simply replacing the $\left(M^{\otimes m}\right)_2$ matrix with the proper one.

\subsection{Margins of multivariate Bernoulli variables with given second-order moments} \label{sec:m1gm2}
In Section \ref{sec:m2gm1} we studied second-order moments of multivariate Bernoulli variables with given margins.   The methodology can be easily generalised to solve the problem of studying $h$-order moments of multivariate Bernoulli variables with given $k$-order moments, $h,k \in\{1,\ldots,m\}, \; h\neq k$. We show this point by studying the $h=1, k=2$ case, i.e. studying margins of multivariate Bernoulli variables $f_{\mu_2}$ with given $2$-order moments $\mu_2=(\mu_{ij}: i,j=1,\ldots,m, \; i < j)$.

We observe that $\expval[X_i X_j]=\sum_{(x_1,\ldots,x_m) \in \design} x_ix_j f_{\mu_2}(x_1,\ldots,x_m)$, that is

\begin{eqnarray*}
x_{ij}^T f_{\mu_2}=\mu_{ij} \\
(1-x_{ij})^T f_{\mu_2} =1-\mu_{ij}
\end{eqnarray*}

where $x_{ij}$ is the vector which contains the product $x_i x_j$ of the $i$-th and the $j$-th element of $x \in \design$. If we consider the odds of the event $X_iX_j=1$, $\gamma_{ij}=\mu_{ij}/(1-\mu_{ij})$, we have $\gamma_{ij} (1-\mu_{ij}) - \mu_{ij}=0$ that is
\[
(\gamma_{ij} (1-x_{ij})^T - x_{ij}^T) f_{\mu_2} = 0.
\]
Building the matrix $H_2$ whose rows are $(\gamma_{ij} (1-x_{ij})^T - x_{ij}^T)$, all the densities $f_{\mu_2}$ must satisty the system of equations $H_2 f_{\mu_2}=0$. The following proposition is the equivalent of Theorem \ref{radii}, Proposition \ref{pr:mu2} and Proposition \ref{pr:bounds} for the case under study.

\begin{proposition}
Let $\ff_{\mu_2}$ a multivariate Bernoulli density with second-order moments $\mu_2=(\mu_{ij}: i,j=1,\ldots,m, \; i < j)$:
\begin{enumerate}
\item all the densities $f_{\mu_2}$ can be written as
\begin{equation}\label{rays}
\ff_{\mu_2}=\sum_{i=1}^{n_{\FFF}}\lambda_i\RRR_{\mu_2}^{(i)},
\end{equation}
where $\RRR_{\mu_2}^{(i)}=(R_{\mu_2}^{(i)}(x), x\in\design)$ $i=1,\dots, n_{\FFF}$ are multivariate Bernoulli densities with second-order moments $\mu_2$, $\lambda_i \geq 0, i=1,\ldots,m$ and $\sum_{i=1}^{n_{\FFF}}\lambda_i=1$.
\item The vector $p=(p_1,\ldots,p_m)$ is
\begin{equation}
p=A_{1\mu_2} \lambda
\end{equation}
where  $A_{1\mu_2} =\left(M^{\otimes m}\right)_1 R_{\mu_2}$ and $\left(M^{\otimes m}\right)_1$ is the sub-matrix of $M^{\otimes m}$ obtained by selecting the rows corresponding to the first-order moments, $R_{\mu_2} $ is the \emph{ray matrix} and $\lambda=(\lambda_1,\ldots,\lambda_{n_{\FFF}})$, $\lambda_i \geq 0, i=1,\ldots m$ and $\sum_{i=1}^{n_{\FFF}}\lambda_i=1$.
\item For each $\alpha$, $\|\alpha\|_0=1$, the first-order moment $\mu_1^{(\alpha)} \equiv p_i$ must satisfy the following bounds
\begin{equation}
\min A_{1\mu_2}^{(\alpha)} \leq p_i \leq \max A_{1\mu_2}^{(\alpha)}
\label{eq:bounds}
\end{equation}
where $A_{1\mu_2}^{(\alpha)}$ is the row of the matrix $A_{1\mu_2} $ such that $p_i=A_{1\mu_2}^{(\alpha)} \lambda$ and
$\{i\}=\{k:\alpha_k=1\}$.
\end{enumerate}
\end{proposition}

\section{Bivariate Bernoulli density with given margins} \label{bivariate}
In this section we consider bivariate distributions, i.e. the class $\FFF(p_1,p_2)$ of $2$-dimensional random variables $(X_1,X_2)$ which have Bernoulli marginal distributions $F_i\sim B(p_i), i=1,2$. In the bivariate case
two key distributions are $F_L$ and $F_u$, the lower and upper Fr\'echet bound of $\FFF(p_1,p_2)$ respectively:
\begin{eqnarray}
F_L(x)=max\{F_1(x_1)+F_2(x_2)-1)  \label{eq:FL} \\
F_U(x)=min\{F_1(x_1), F_2(x_2)\} \label{eq:FU}
\end{eqnarray}
where $x=(x_1,x_2)\in\{0,1\}^2$.

For any  $F_p\in \FFF(p_1, p_2)$ it holds that
\begin{equation} \label{eq:Fbounds}
F_L(x)\leq F_p(x)\leq F_U(x), \;  x\in\{0,1\}^2.
\end{equation}
For an overview of Fr\'echet classes and their bounds see \cite{dall2012advances}.

We now analyse Theorem \ref{radii} in the bivariate case. The number of rays is independent of the Fr\'echet class $\FFF(p_1,p_2)$. We have two ray densities, which are the lower and upper Fr\'echet bound of each class.

\begin{proposition} \label{pr:biv}
Let $\ff\in \mathcal{F}(p_1, p_2)$, then
\begin{equation*}
\ff_p=\lambda\ff_L+(1-\lambda)\ff_U,\,\,\ \lambda\in[0,1],
\end{equation*}
where   $f_L$ and $f_U$ are the discrete densities corresponding to $F_L$ and $F_U$, respectively.
\end{proposition}

\begin{proof}
We observe that in $x=(0,0)$ the distribution function and the density function take the same value. Then using
\eqref{eq:Fbounds} we can write
\begin{equation} \label{eq:b00}
f_L(0,0) \leq f_p(0,0) \leq f_U(0,0) .
\end{equation}

It follows that $f_p(0,0)=\lambda f_L(0,0) + (1-\lambda) f_U(0,0)$ with $\lambda=\frac{f_p(0,0)-f_U(0,0)}{f_L(0,0)-f_U(0,0)}$. It holds that $0 \leq \lambda \leq 1$.

Now we observe that for any density function $f \in \FFF(p_1,p_2)$ we have $f(0,1)=q_1 - f(0,0)$. Then using \eqref{eq:b00} we can write
\[
q_1-f_L(0,0) \geq q_1-f_p(0,0) \geq q_1-f_U(0,0)
\]
that is
\[
f_U(1,0) \leq f_p(1,0) \leq f_L(1,0).
\]
We can write $f_p(1,0)=\lambda_1 f_L(1,0) + (1-\lambda_1) f_U(1,0)$. It is easy to verify that $\lambda_1=\lambda$.
We proceed in an analogous way for $f_p(0,1)=q_2-f_p(0,0)$ and $f_p(1,1)=1-q_1-q_2+f_p(0,0)$ and we get $f_p(x)=\lambda f_L(x)+(1-\lambda) f_U(x)$, $x \in \{0,1\}^2$ and $0 \leq \lambda \leq 1$.
\end{proof}

Proposition \ref{pr:biv} states that $\mathcal{F}(p_1, p_2)$ is the convex hull of the upper and lower Fr\'echet bound.

In the bivariate case we can also find the domain of $\theta_{12}$ expressed as a function of the margins $p_1, p_2$. From Eq.\eqref{dens} we get
\begin{equation}\label{dens2b}
f_p(0,0)= q_1q_2(1+\theta_{12}p_1p_2).
\end{equation}
and consequently
\begin{equation}\label{ttt}
\theta_{12}=\frac{f_p(0,0)-q_1q_2}{q_1q_2p_1p_2}.
\end{equation}
Using \eqref{eq:b00} it follows
\[
\frac{f_L(0,0)-q_1q_2}{q_1q_2p_1p_2} \leq \theta_{12} \leq \frac{f_U(0,0)-q_1q_2}{q_1q_2p_1p_2}
\]

Now without loss of generality we assume $q_2\geq q_1$.
From Eq.\eqref{eq:FL} and \eqref{eq:FU} we get
\begin{enumerate}
\item if $q_1+q_2 \leq 1$ then $-\frac{1}{p_1p_2} \leq \theta_{12} \leq \frac{1}{p_1q_2}$;
\item if $q_1+q_2 >1$ then $\frac{q_1+q_2-1-q_1q_2}{q_1q_2p_1p_2} \leq \theta_{12} \leq \frac{1}{p_1q_2}$.
\end{enumerate}

Finally (see also Theorem 1 in \cite{huber2015multivariate}) we obtain the bounds for the correlation coefficient
\[
\rho_{12}=\frac{\expval [X_1X_2]-p_1p_2}{\sqrt{p_1q_1p_2q_2}}.
\]
Being $\expval[X_1X_2]=f_p(1,1)$, $f_L(1,1) \leq f_p(1,1) \leq f_U(1,1)$ and $f(1,1)=1-q_1-q_2+f(0,0)$ for any density function $f \in \FFF(p_1,p_2)$ we obtain:
\begin{enumerate}
\item if $q_1+q_2 \leq 1$ then $\frac{1-q_1-q_2-p_1p_2}{\sqrt{p_1q_1p_2q_2}} \equiv - \sqrt{\frac{q_1q_2}{p_1p_2}} \leq \rho_{12} \leq \frac{1-q_2-p_1p_2}{\sqrt{p_1q_1p_2q_2}} \equiv \sqrt{\frac{p_2q_1}{p_1q_2}}$;
\item if $q_1+q_2 >1$ then $-\frac{p_1p_2}{\sqrt{p_1q_1p_2q_2}} \equiv -\sqrt{\frac{p_1p_2}{q_1q_2}}\leq \rho_{12} \leq \frac{1-q_2-p_1p_2}{\sqrt{p_1q_1p_2q_2}} \equiv \sqrt{\frac{p_2q_1}{p_1q_2}}$.
\end{enumerate}


\section{Examples}\label{sec:examples}
In this section we show some results corresponding to different multivariate Bernoulli distributions. The algorithm is described in Section \ref{algo}.

\subsection{Trivariate Bernoulli distributions}
Let us consider the case $m=3$ and $p=\left(\frac{1}{2},\frac{1}{2},\frac{1}{2}\right)$. From Theorem \ref{radii}, solving the system of equations \eqref{H}, we get $6$ \emph{ray densities}. The \emph{ray} matrix $R_p$ is
\[
R_p=\left(
\begin{array}{rrrrrr}
0 &	0 &	0 &	0 &	0.5 &	0.25 \\
0 &	0 &	0.5 &	0.25 &	0 &	0 \\
0 &	0.5 &	0 &	0.25 &	0 &	0 \\
0.5 &	0 &	0 &	0 &	0 &	0.25 \\
0.5 &	0 &	0 &	0.25 &	0 &	0 \\
0 &	0.5 &	0 &	0 &	0 &	0.25 \\
0 &	0 &	0.5 &	0 &	0 &	0.25 \\
0 &	0 &	0 &	0.25 &	0.5 &	0 \\
\end{array}
\right)
\]
and the matrix $A_{2p}$ as defined in Proposition \ref{pr:mu2} is
\[
A_{2p}=\left(
\begin{array}{rrrrrr}
0.5 &	0 &	0 &	0.25 &	0.5 &	0.25 \\
0 &	0.5 &	0 &	0.25 &	0.5 &	0.25 \\
0 &	0 &	0.5 &	0.25 &	0.5 &	0.25 \\
\end{array}
\right).
\]
Using Eq. \eqref{eq:bounds} we get
\[
-1 \leq \rho_{ij} \leq 1, \; i,j=1,2,3, i<j.
\]

Let us consider the case in which the $X_i, i=1,\ldots,3$ must be not correlated. We want to find a distribution $F_p \in \mathcal{F}(\frac{1}{2},\frac{1}{2},\frac{1}{2})$ such that $\rho_{12}=\rho_{13}=\rho_{23}=0$.
From Eq. \eqref{eq:rho} we obtain $\lambda_1=\lambda_2=\lambda_3=\lambda_5=0.25$ and $\lambda_4=\lambda_6=0$. The corresponding density is uniform, $f_p(x)=\frac{1}{8}, x \in \mathcal{S}_3$ as expected.

If we choose $\rho_{12}=0.2, \rho_{13}=-0.3$ and $\rho_{23}=0.4$, we obtain $\lambda_1=0.275,
\lambda_2=0.025, \lambda_3=0.375, \lambda_4=0, \lambda_5=0.325$ and $\lambda_6=0$ as one of the solutions of Eq. \eqref{eq:rho}.  The corresponding density is
\[
\ff_p^T=
\left(0.1625, \; 0.1875, \; 0.0125, \; 0.1375, \; 0.1375, \; 0.0125, \; 0.1875, \; 0.1625
\right).
\]
If we choose $\rho_{12}=0.9, \rho_{13}=-0.3$ and $\rho_{23}=0.6$, we do not find any $f_p$ with such correlations, even if each $\rho_{ij}$ satisfies the constraints found for bivariate distributions, which, as we said before, in this case are $-1\leq\rho_{ij}\leq1, \, i,j=1,2,3, \, i < j$.

If we search for a feasible $\rho^\star$ which is the correlation matrix  closest\footnote{The distance can be freely chosen. In this example we used the Euclidean distance.} to the desired $\rho$ we obtain $\rho_{12}^\star=0.6\overline{3}$, $\rho_{13}^\star=0.3\overline{3}$ and $\rho_{23}^\star=-0.0\overline{3}$. The corresponding density is
\[
(\ff_p^\star)^{T}=
\left(
0.241\overline{6}, \;
0, \;
0.091\overline{6}, \;
0.166\overline{6}, \;
0.166\overline{6}, \;
0.091\overline{6}, \;
0, \;
0.241\overline{6}
\right).
\]
Let us now consider the case $p=\left(\frac{1}{4},\frac{3}{4},\frac{1}{2}\right)$. The ray matrix $R_p$ contains $6$ margins
\[R_p=\left(
\begin{array}{rrrrrr}
0 &	0 &	0 &	0 &	0.25 &	0.25 \\
0.25 &	0.25 &	0.5 &	0.5 &	0 &	0.25 \\
0 &	0.25 &	0 &	0 &	0 &	0 \\
0.25 &	0 &	0 &	0 &	0.25 &	0 \\
0.25 &	0 &	0 &	0.25 &	0 &	0 \\
0.25 &	0.5 &	0.25 &	0 &	0.5 &	0.25 \\
0 &	0 &	0.25 &	0 &	0 &	0 \\
0 &	0 &	0 &	0.25 &	0 &	0.25 \\
\end{array}
\right)
\]
and the $A_{2p}$ matrix is
\[A_{2p}=\left(
\begin{array}{rrrrrr}
0.25 &	0 &	0 &	0.25 &	0.25 &	0.25 \\
0.25 &	0.5 &	0.25 &	0.25 &	0.5 &	0.5 \\
0 &	0 &	0.25 &	0.25 &	0 &	0.25 \\
\end{array}
\right).
\]
Using Eq. \eqref{eq:bounds} we get
\[
-1 \leq \rho_{12} \leq 0.333 \text{ and } -0.577 \leq \rho_{13},\rho_{23} \leq 0.577.
\]

If we choose $\rho_{12}=0.3, \rho_{13}=0.25$ and $\rho_{23}=-0.1$, we obtain $\lambda_1=0.2835,
\lambda_2=0.025, \lambda_3=0, \lambda_4=0, \lambda_5=0.2781$ and $\lambda_6=0.4134$.  The corresponding density is
\[
\ff_p^T=
\left(
0.1729, \;
0.1805, \;
0.0063, \;
0.1404, \;
0.0709, \;
0.3258, \;
0, \;
0.1033\;
\right).
\]
As the last example of trivariate Bernoulli distribution we consider $p=\left(\frac{1}{4},\frac{1}{7},\frac{1}{3}\right)$. The ray matrix $R_p$ (rounded to the third decimal digit) has $11$ \emph{ray densities}
\[
R_p=\left(
\begin{array}{rrrrrrrrrrr}
0 &	0 &	0 &	0 &	0 &	0 &	0 &	0 &	0.06 &	0.143 &	0.143 \\
0 &	0 &	0 &	0 &	0.083 &	0.143 &	0.143 &	0.113 &	0.083 &	0 &	0 \\
0 &	0 &	0.107 &	0.25 &	0.25 &	0 &	0.19 &	0.22 &	0.19 &	0 &	0.107 \\
0.333 &	0.333 &	0.226 &	0.083 &	0 &	0.19 &	0 &	0 &	0 &	0.19 &	0.083 \\
0 &	0.143 &	0.143 &	0 &	0 &	0 &	0 &	0.03 &	0 &	0 &	0 \\
0.143 &	0 &	0 &	0.143 &	0.06 &	0 &	0 &	0 &	0 &	0 &	0 \\
0.25 &	0.107 &	0 &	0 &	0 &	0.25 &	0.06 &	0 &	0 &	0.107 &	0 \\
0.274 &	0.417 &	0.524 &	0.524 &	0.607 &	0.417 &	0.607 &	0.637 &	0.667 &	0.56 &	0.667 \\
\end{array}
\right).
\]
Using Eq. \eqref{eq:bounds} we get
\[
-0.236 \leq \rho_{12} \leq 0.707, -0.408 \leq \rho_{13} \leq 0.816 \text{ and } -0.289 \leq \rho_{23} \leq 0.577.
\]


If we choose $\rho_{12}=0.3, \rho_{13}=0.25$ and $\rho_{23}=-0.2$, we obtain 
\[
\ff_p^T=
\left(
0.0146, \;
0, \;
0.1197, \;
0.1990, \;
0.0665, \;
0.0617, \;
0.0491, \;
0.4893
\right).
\]

\subsection{Multivariate $m=5$ Bernoulli distributions}
Let us consider the case $p=\left(\frac{1}{2},\frac{1}{2},\frac{1}{2},\frac{1}{2},\frac{1}{2}\right)$. We obtain $2,712$ {\it  ray densities}.
If we choose
$\rho_{12}=0.3, \rho_{13}= 0.2, \rho_{14}=0.2, \rho_{15}=0.1,
\rho_{23}=-0.2, \rho_{24}=0.3, \rho_{25}=0.2,
\rho_{34}=0.2, \rho_{35}=0.1$ and $\rho_{45}= -0.2$, we obtain
{\tiny
\[
\ff_p=
\left(
\begin{array}{r}
0.025 \\
0 \\
0.0625 \\
0.0125 \\
0 \\
0.025 \\
0.025 \\
0.05 \\
0.1 \\
0.025 \\
0 \\
0.05 \\
0.0875 \\
0.0375 \\
0 \\
0 \\
0.1 \\
0.05 \\
0.0125 \\
0.0625 \\
0 \\
0 \\
0.05 \\
0.025 \\
0 \\
0 \\
0 \\
0 \\
0.0125 \\
0.0375 \\
0.025 \\
0.125 \\
\end{array}
\right).
\]
}
\subsection{Multivariate $m\geq6$ Bernoulli distributions}
For $m=6$ and $p=\left(\frac{1}{2},\frac{1}{2},\frac{1}{2},\frac{1}{2},\frac{1}{2},\frac{1}{2}\right)$ we obtain $707,264$ {\it  ray densities}. In general we observe that if the number of rays is too large with respect to the available computer power and if the objective can be reduced to the problem of finding just \emph{one} density $f \in \mathbb{F}_m$ with given margins $p$ and second order moments $\mu_2$, it is enough to solve the system
\[
\begin{cases}
(M^{\otimes m})_{1} f = p \\
(M^{\otimes m})_{2} f = \mu_2
\end{cases}
\]
using standard linear programming tools (e.g. \cite{berkelaar2004lpsolve}).

\subsection{The algorithm}\label{algo}
In this section we briefly describe the algorithm that we used in Section \ref{sec:examples}. Given $m$, $p$ and $\rho$ as input the algorithm returns the ray matrix $R_p$ and, if it exists, the density $f_p$, which has Bernoulli $B(p_i), i=1,\ldots,m$ as marginal distribution and pairwise correlations $\rho=(\rho_{ij}, i,j=1,\ldots,m, i<j)$. The algorithm has the following main steps:
\begin{enumerate}
	\item the construction of the matrix $H$, see \eqref{H} of Theorem \ref{radii};
	\item the generation of the \emph{ray matrix} $R_p$;
	\item the construction of the density $f_p$ as the solution of the system \eqref{eq:rho} of Theorem \ref{radii}.
\end{enumerate}

The construction of the matrix $H$ and of the density $f_p$ is implemented in SAS/IML. In particular, the system \eqref{eq:rho} is solved using the Proc Lpsolve that is part of SAS/QC. The rays are generated using 4ti2 (\cite{4ti2}). The software code is available on request. We performed the analysis using a standard laptop (CPU Intel core I7-2620M CPU 2.70GHz 2.70GHz, RAM 8GB).

\section{Discussion}\label{Concl}
The proposed approach can be applied to \emph{any} given set of moments, even of different orders.  All the results given for moments and correlations can be easily adapted to other widely-used measures of dependence, such as Kendall's $\tau$ and Spearman's $\rho$ . Furthermore,  the polynomial representation of the distributions of any Fr\'echet class provides a link to copulas, which are a powerful instrument to model dependence.

\section{Acknowledgements}
Roberto Fontana wishes to thank professor Antonio Di Scala (Politecnico di Torino, Department of Mathematical Sciences) and professor Giovanni Pistone (Collegio Carlo Alberto, Moncalieri) for the helpful discussions he had with them.

\bibliographystyle{plain}

\end{document}